\documentclass[11pt]{amsart}

\usepackage{amsmath}
\usepackage{fullpage}
\usepackage{xspace}
\usepackage[psamsfonts]{amssymb}
\usepackage[latin1]{inputenc}
\usepackage{graphicx,color}
\usepackage[curve]{xypic} 
\usepackage{hyperref}
\usepackage{graphicx}
\usepackage{amsmath}%
\usepackage{amsthm}%
\usepackage{amscd}
\usepackage{amsfonts}%
\usepackage{amssymb}%
\usepackage{graphicx}

\usepackage{mathrsfs}

\usepackage{tikz}
\usetikzlibrary{matrix,arrows}

\usepackage{tikz-cd}

\newtheorem{theorem}{Theorem}[section]

\newtheorem{conjecture}[theorem]{Conjecture}

\newtheorem{corollary}[theorem]{Corollary}

\newtheorem{lemma}[theorem]{Lemma}

\theoremstyle{remark}

\numberwithin{equation}{section}

\newcommand{\Z}{\mathbb{Z}}

\newcommand{\Q}{\mathbb{Q}}

\newcommand{\den}{\mathrm{den}}
\newcommand{\num}{\mathrm{num}}

\makeatletter
\@namedef{subjclassname@2020}{%
  \textup{2020} Mathematics Subject Classification}
\makeatother

  \DeclareFontFamily{U}{wncy}{}
    \DeclareFontShape{U}{wncy}{m}{n}{<->wncyr10}{}
    \DeclareSymbolFont{mcy}{U}{wncy}{m}{n}
    \DeclareMathSymbol{\Sha}{\mathord}{mcy}{"58}

\begin{document}
\title[]{Szpiro's conjecture when the denominator of the $j$-invariant is small}

\author{Hector Pasten}
\address{ Departamento de Matem\'aticas,
Pontificia Universidad Cat\'olica de Chile.
Facultad de Matem\'aticas,
4860 Av.\ Vicu\~na Mackenna,
Macul, RM, Chile}
\email[H. Pasten]{hpasten@gmail.com}%

%\thanks{}
\thanks{Supported by ANID Fodecyt Regular grant 1230507 from Chile.}
\date{\today}
\subjclass[2020]{Primary: 11G05; Secondary: 11G50} %
\keywords{Szpiro's conjecture, elliptic curves, discriminant, conductor}%
\dedicatory{Dedicated to M. Ram Murty in the occasion of his 70th birthday. Happy birthday Professor Murty!}

\begin{abstract} We prove Szpiro's conjecture for elliptic curves over the rationals having $j$-invariant with denominator of logarithmic size with respect to its numerator.
\end{abstract}

\maketitle

%\tableofcontents

%%%%%%%%%%%%%%%%%%%%%%%%%%%%%%%%%%%%%%
%%%%%%%%%%%%%%%%%%%%%%%%%%%%%%%%%%%%%%
%%%%%%%%%%%%%%%%%%%%%%%%%%%%%%%%%%%%%%
%%%%%%%%%%%%%%%%%%%%%%%%%%%%%%%%%%%%%%
%%%%%%%%%%%%%%%%%%%%%%%%%%%%%%%%%%%%%%
%%%%%%%%%%%%%%%%%%%%%%%%%%%%%%%%%%%%%%

\section{Introduction} 

 We begin by fixing some notation. For an elliptic curve $E$ over $\Q$ we write $j_E$ for its $j$-invariant, $\Delta_E$ for the absolute value of its minimal discriminant, and $N_E$ for its conductor.

In the early 80's, Szpiro proposed the following conjecture \cite{Szpiro}:
%%%
%%%
\begin{conjecture}[Szpiro's conjecture] Let $\epsilon>0$. There is a number $c_\epsilon>0$ depending only on $\epsilon$ such that for all elliptic curves $E$ over $\Q$ we have $\Delta_E \le c_\epsilon \cdot N_E^{6+\epsilon}$.
\end{conjecture}
%%%
%%%

This conjecture is very deep. Even a weaker version with the exponent $6+\epsilon$ replaced by some fixed constant would have tremendous consequences such as a version of the $abc$-conjecture ---in fact, Szpiro's conjecture was the main motivation for the formulation of the $abc$-conjecture, see \cite{Masser}.

 Szpiro's conjecture is known for elliptic curves of prime power discriminant by work of Mestre and Oesterl\'e \cite{MestreOesterle} and for elliptic curves of integral $j$-invariant by work of Pesenti and Szpiro \cite{PesentiSzpiro}. 
 
 At present, the strongest unconditional result valid for all elliptic curves over $\Q$ is the following (effective) estimate by the author \cite{PastenShimuraCurves} valid for any $\epsilon>0$: $\log \Delta_E \le (1/4 +\epsilon )\cdot N_E\log N_E + O_\epsilon(1)$.  This improves the earlier bound $\log \Delta_E \le N_E\log N_E+O(N_E\log \log N_E)$ by Murty and the author \cite{MurtyPasten}; both results use the theory of modular forms. 

We mention that an upper bound for the number of exceptions to Szpiro's conjecture with $\Delta_E$ less than a given bound is proved in \cite{FNT} by Fouvry, Nair, and Tenenbaum.

%%%
%%%

 Our goal is to prove Szpiro's conjecture for elliptic curves whose $j$-invariant has small denominator. We write $\num(q)$ and $\den(q)$ for the absolute value of the numerator and denominator of a rational number $q$ in reduced form. With this notation, our main result is:

\begin{theorem}[Main Result]\label{ThmMain} Let $A,B>0$. For all elliptic curves $E$ over $\Q$ with $\den(j_E)\le A(\log \num(j_E))^B$ we have  $\Delta_E\le A\cdot 16^{B+1} N_E^{B+5}(\log N_E)^B$.
\end{theorem}

In particular, Szpiro's conjecture holds when $\den(j_E)$ has logarithmic size with respect to $\num(j_E)$:

\begin{corollary}[Szpiro's conjecture when $\den(j_E)$ is small] Let $A>0$. For all elliptic curves $E$ over $\Q$ with $\den(j_E) \le A\cdot \log \num(j_E)$ we have $\Delta_E \le 256A\cdot  N_E^6\log N_E$.
\end{corollary}
 
Note that this generalizes the Pesenti--Szpiro result \cite{PesentiSzpiro} on Szpiro's conjecture when $j_E\in \Z$.

Finally, let us mention an application. Theorem \ref{ThmMain} together with Theorem 0.7 in \cite{HindrySilverman} by Hindry and Silverman yield a rather uniform bound for the number of $S$-integral points on elliptic curves $E$ over $\Q$ whenever $\den(j_E)\le (\log \num (j_E)))^B$ for a fixed $B$. (Nevertheless, it is likely that the latter condition can be wakened for this application by revisiting ideas from Silverman's thesis.)

%%%%%%%%%%%%%%%%%%%%%%%%%%%%%%%%%%%%%%
%%%%%%%%%%%%%%%%%%%%%%%%%%%%%%%%%%%%%%
%%%%%%%%%%%%%%%%%%%%%%%%%%%%%%%%%%%%%%
%%%%%%%%%%%%%%%%%%%%%%%%%%%%%%%%%%%%%%
%%%%%%%%%%%%%%%%%%%%%%%%%%%%%%%%%%%%%%
%%%%%%%%%%%%%%%%%%%%%%%%%%%%%%%%%%%%%%

\section{The height of the $j$-invariant} 

The Faltings height of an elliptic curve $E$ over $\Q$ is denoted by $h(E)$. In \cite{MurtyPasten}, Murty and the author used the theory of modular forms to prove the following  explicit bound for all $E$ over $\Q$:
$$
h(E) < 0.1\cdot N_E\log N_E + 11.
$$

For a rational number $q$ we recall that its logarithmic height is $h(q)=\log \max\{\num(q),\den(q)\}$. It turns out that  $h(E)$ is related to $h(j_E)$ in a very explicit way; Silverman \cite{SilvermanE} proved
$$
h(j_E) \le 12 h(E) + O(\log(2+h(j_E)))
$$
where the error term has an effective implicit constant. This has been made explicit by Pellarin \cite{Pellarin} and, in a sharper form, by Pazuki \cite{Pazuki}. For our purposes Lemme 5.2 in \cite{Pellarin} is enough; this gives $h(j_E)\le 24\max\{1, h(E)\} + 94.3$. Putting these results together we obtain:

\begin{corollary}\label{CoroMain1} For all elliptic curves $E$ over $\Q$ we have $h(j_E) \le 16 \cdot N_E\log N_E$.
\end{corollary}
\begin{proof} The previous discussion leads to
$$
h(j_E) \le 94.3 + 24\cdot (0.1\cdot N_E\log N_E + 11) = 2.4 \cdot N_E\log N_E + 358.3.
$$
The result follows from the well-known fact that $N_E\ge 11$ for all elliptic curves over $\Q$.
\end{proof}

%%%%%%%%%%%%%%%%%%%%%%%%%%%%%%%%%%%%%%
%%%%%%%%%%%%%%%%%%%%%%%%%%%%%%%%%%%%%%
%%%%%%%%%%%%%%%%%%%%%%%%%%%%%%%%%%%%%%
%%%%%%%%%%%%%%%%%%%%%%%%%%%%%%%%%%%%%%
%%%%%%%%%%%%%%%%%%%%%%%%%%%%%%%%%%%%%%
%%%%%%%%%%%%%%%%%%%%%%%%%%%%%%%%%%%%%%

\section{An application of Tate's algorithm} 

For a prime number $p$ we denote by $v_p:\Q^\times\to \Z$ the $p$-adic valuation.

Let $E$ be an elliptic curve over $\Q$. The primes $p$ that divide $\Delta_E$ are the same as the ones that divide $N_E$. We split these primes $p$ into three types:

\begin{itemize}
\item \emph{Type 1.} $v_p(j_E)\ge 0$.
\item \emph{Type 2.} $v_p(j_E)<0$ and $E$ has multiplicative reduction at $p$.
\item \emph{Type 3.} $v_p(j_E)<0$ and $E$ has additive reduction at $p$.
\end{itemize}

The following is proved in \cite{PesentiSzpiro}.

\begin{lemma}[Primes of Type 1: Pesenti--Szpiro] If $p$ is of Type 1, then $v_p(\Delta_E)\le 5 v_p(N_E)$.
\end{lemma}

As noted in \cite{PesentiSzpiro}, this immediately implies Szpiro's conjecture whenever $j_E\in \Z$.

Let us now consider $p$ of Type 2. The Kodaira type of the special fibre of the minimal regular model of $E$ over $\Z_p$ is $I_n$ for some $n\ge 1$. The output of Tate's algorithm summarized in Table 4.1 in p. 365 of \cite{SilvermanAdvanced} (which refined Tate's table in \cite{Tate}) shows that $n=-v_p(j_E)=v_p(\Delta_E)$. So we find:

\begin{lemma}[Primes of Type 2] If $p$ is of Type 2, then $v_p(\Delta_E) = - v_p(j_E)$.
\end{lemma}

Finally we deal with the primes $p$ of Type 3. In this case the Kodaira type of $E$ at $p$ is $I_n^*$ for certain integer $n\ge 1$, see the discussion in p.42 of \cite{Lorenzini}. 

If $p\ge 3$, from the data in the $I_n^*$ column of Table 4.1 in p. 365 of \cite{SilvermanAdvanced} we get
$$
v_p(\Delta_E) = 6+n = 6-v_p(j_E) = 3v_p(N_E)-v_p(j_E).
$$

Let us now consider the case $p=2$. The number $m$ of geometrically irreducible components of the special fibre of the minimal regular model at $p=2$ is $m=n+5$ (see the table in \cite{Tate}.) By Ogg's formula we have $v_2(N_E)=v_2(\Delta_E)-m+1 = v_2(\Delta_E)-n-4$ which gives $v_2(\Delta_E)=v_2(N_E) + 4+ n$. 

We need some control on the integer $n$. Theorem 2.8 in \cite{Lorenzini} gives the existence of a suitable quadratic extension $L/\Q$ such that if $s+1$ is the valuation of its different ideal over $2$, then $n=-v_2(j_E) + 4s$.  By Remark 1 in p.58 of \cite{SerreLF} we have $s\le 2$ so we get $n\le -v_2(j_E)+8$. Therefore
$$
v_2(\Delta_E)\le v_2(N_E) -v_2(j_E)+12 \le 3v_2(N_E) - v_2(j_E) + 8
$$
because $v_2(N_E)\ge 2$ (additive reduction). Let us summarize our findings:

\begin{lemma}[Primes of Type 3] Let  $p$ be a prime of Type 3 and write $\delta_p=8$ if $p=2$ and $\delta_p=0$ if $p\ge 3$. Then $v_p(\Delta_E)\le 3v_p(N_E) - v_p(j_E) + \delta_p$.
\end{lemma}

From these three lemmas we deduce the following result, which can be of independent interest:
\begin{corollary}\label{CoroMain2} Let $E$ be an elliptic curve over $\Q$. Then $\Delta_E$ divides $16\cdot \den(j_E)N_E^5$.
\end{corollary}
%%
%%

%%%%%%%%%%%%%%%%%%%%%%%%%%%%%%%%%%%%%%
%%%%%%%%%%%%%%%%%%%%%%%%%%%%%%%%%%%%%%
%%%%%%%%%%%%%%%%%%%%%%%%%%%%%%%%%%%%%%
%%%%%%%%%%%%%%%%%%%%%%%%%%%%%%%%%%%%%%
%%%%%%%%%%%%%%%%%%%%%%%%%%%%%%%%%%%%%%
%%%%%%%%%%%%%%%%%%%%%%%%%%%%%%%%%%%%%%

\section{Proof of the Main Result} 

\begin{proof}[Proof of Theorem \ref{ThmMain}] By Corollary \ref{CoroMain1} we have
$$
\den(j_E)\le A(\log \num(j_E))^B\le Ah(j_E)^B\le A\cdot 16^B N_E^B(\log N_E)^B.
$$
Putting this estimate together with Corollary \ref{CoroMain2} we find $\Delta_E\le A\cdot 16^{B+1} N_E^{B+5}(\log N_E)^B$. 
\end{proof}

%%%%%%%%%%%%%%%%%%%%%%%%%%%%%%%%%%%%%%
%%%%%%%%%%%%%%%%%%%%%%%%%%%%%%%%%%%%%%
%%%%%%%%%%%%%%%%%%%%%%%%%%%%%%%%%%%%%%
%%%%%%%%%%%%%%%%%%%%%%%%%%%%%%%%%%%%%%
%%%%%%%%%%%%%%%%%%%%%%%%%%%%%%%%%%%%%%
%%%%%%%%%%%%%%%%%%%%%%%%%%%%%%%%%%%%%%
%%%%%%%%%%%%%%%%%%%%%%%%%%%%%%%%%%%%%%
%%%%%%%%%%%%%%%%%%%%%%%%%%%%%%%%%%%%%%
%%%%%%%%%%%%%%%%%%%%%%%%%%%%%%%%%%%%%%
%%%%%%%%%%%%%%%%%%%%%%%%%%%%%%%%%%%%%%
%%%%%%%%%%%%%%%%%%%%%%%%%%%%%%%%%%%%%%
%%%%%%%%%%%%%%%%%%%%%%%%%%%%%%%%%%%%%%

\section{Acknowledgments}

Supported by ANID Fondecyt Regular grant 1230507 from Chile. I thank Joseph Silverman for answering a question on Tate's algorithm and for useful comments on an earlier version of this note. I also thank Natalia Garcia-Fritz and Fabien Pazuki for suggesting some improvements in the presentation.

%%%%%%%%%%%%%%%%%%%%%%%%%%%%%%%%%%%%%%
%%%%%%%%%%%%%%%%%%%%%%%%%%%%%%%%%%%%%%
%%%%%%%%%%%%%%%%%%%%%%%%%%%%%%%%%%%%%%


\begin{thebibliography}{9}         

\bibitem{FNT} \'E. Fouvry, M. Nair, G. Tenenbaum, \emph{L'ensemble exceptionnel dans la conjecture de Szpiro}.  Bull. Soc. Math. France 120 (1992), no. 4, 485-506. 


\bibitem{HindrySilverman} M. Hindry, J. Silverman, \emph{The canonical height and integral points on elliptic curves}. Invent. Math. 93 (1988), no. 2, 419-450.

\bibitem{Lorenzini} D. Lorenzini, \emph{Models of curves and wild ramification}. Pure Appl. Math. Q. 6 (2010), no. 1, Special Issue: In honor of John Tate. Part 2, 41-82. 


\bibitem{Masser} D. Masser, \emph{Abcological anecdotes}. Mathematika 63 (2017), no. 3, 713-714. 

\bibitem{MestreOesterle} J.-F. Mestre, J. Oesterl\'e, \emph{Courbes de Weil semi-stables de discriminant une puissance m-i\`eme}. J. Reine Angew. Math. 400 (1989), 173-184.

\bibitem{MurtyPasten} M. R. Murty, H. Pasten, \emph{Modular forms and effective Diophantine approximation}. J. Number Theory 133 (2013), no. 11, 3739-3754.


\bibitem{PastenShimuraCurves} H. Pasten, \emph{Shimura curves and the $abc$ conjecture}. Journal of Number Theory, Prime Section. To appear (2023).

\bibitem{Pazuki} F. Pazuki, \emph{Modular invariants and isogenies}. Int. J. Number Theory 15 (2019), no. 3, 569-584.

\bibitem{Pellarin} F. Pellarin, \emph{Sur une majoration explicite pour un degr\'e d'isog\'enie liant deux courbes elliptiques}. Acta Arith. 100 (2001), no. 3, 203-243. 

\bibitem{PesentiSzpiro} J. Pesenti, L. Szpiro, \emph{In\'egalit\'e du discriminant pour les pinceaux elliptiques \`a r\'eductions quelconques}.  Compositio Math. 120 (2000), no. 1, 83-117.


\bibitem{SerreLF} J.-P. Serre, \emph{Local fields}. Translated from the French by Marvin Jay Greenberg. Graduate Texts in Mathematics, 67. Springer-Verlag, New York-Berlin, 1979. viii+241 pp. ISBN: 0-387-90424-7


\bibitem{SilvermanE} J. Silverman, \emph{Heights and elliptic curves}. Arithmetic geometry (Storrs, Conn., 1984), 253-265, Springer, New York, 1986.

\bibitem{SilvermanAdvanced} J. Silverman, \emph{Advanced topics in the arithmetic of elliptic curves}. Graduate Texts in Mathematics, 151. Springer-Verlag, New York, 1994. xiv+525 pp. ISBN: 0-387-94328-5

\bibitem{Szpiro} L. Szpiro, \emph{Pr\'esentation de la th\'eorie d'Arak\'elov}. Current trends in arithmetical algebraic geometry (Arcata, Calif., 1985), 279-293, Contemp. Math., 67, Amer. Math. Soc., Providence, RI, 1987.

\bibitem{Tate} J. Tate, \emph{Algorithm for determining the type of a singular fiber in an elliptic pencil}. Modular functions of one variable, IV (Proc. Internat. Summer School, Univ. Antwerp, Antwerp, 1972), pp. 33-52, Lecture Notes in Math., Vol. 476, Springer, Berlin-New York, 1975.

\end{thebibliography}
\end{document}